\documentclass{amsart}
\usepackage{graphicx}
\usepackage{amsmath,amsfonts}
\usepackage{amssymb,latexsym}
\newtheorem{theorem}{Theorem}[section]
\newtheorem{corollary}[theorem]{Corollary}
\newtheorem{lemma}[theorem]{Lemma}
\newtheorem{proposition}[theorem]{Proposition}
\theoremstyle{definition}

\theoremstyle{remark}
\newtheorem{remark}[theorem]{Remark}
\newtheorem{example}[theorem]{Example}
\numberwithin{equation}{section}
\newcommand{\F}{\mathbb F}
\newcommand{\G}{\mathbb G}

\newcommand{\A}{\mathbb A}
\newcommand{\bfs}{\boldsymbol}

\newcommand{\fq}{\F_{\hskip-0.7mm q}}
\newcommand{\cfq}{\overline{\F}_{\hskip-0.7mm q}}

\def\ifm#1#2{\relax \ifmmode#1\else#2\fi}

\newcommand{\klk}    {\ifm {,\ldots,} {$,\ldots,$}}

\begin{document}

\title[]{Entropy approximations of algebraic matroids over finite fields}%
\author[G. Matera]{
Guillermo Matera${}^{1,2}$}%
\address{${}^{1}$ National Council of Science and Technology (CONICET),
Ar\-gentina}
\address{${}^{2}$Instituto del Desarrollo Humano,
Universidad Nacional de Gene\-ral Sarmiento, J.M. Guti\'errez 1150
(B1613GSX) Los Polvorines, Buenos Aires, Argentina}
\email{\{gmatera\}@ungs.edu.ar}

\thanks{The author is partially supported by the grant
PIP CONICET 11220130100598}%
\subjclass{05B35 (Primary); 94A15, 14G15, 11G25 (Secondary)}%
\keywords{Algebraic matroids; Entropic polymatroids; Shannon
entropy; Algebraic geometry; Lang-Weil estimates; Secret sharing;
Network coding.
}%

\begin{abstract}
We investigate the asymptotic behavior of entropy polymatroids
associated with algebraic matroids over finite fields. Given an
algebraic matroid ${\sf M}:=(\mathcal{E},r)$ and the irreducible variety $V$
associated with ${\sf M}$, we consider the polymatroid $h_\F$ induced by the
entropies of the projections of $V(\F)$, where $\F$ is a finite extension
of $\G$. Revisiting a construction of Mat\'u$\check{\mathrm{s}}$, we show that
the polymatroid $h_\F$ converges to the rank function $r$ of ${\sf M}$
as $q:=|\F|$ tends to infinity.

Our main contribution is to make this convergence quantitative: we
derive explicit uniform error bounds for the deviation $|h_\F-r|$,
expressed in terms of the degree of $V$, the ground set size
$|\mathcal{E}|$, the rank $r(\mathcal{E})$, and $q$. The proofs
combine tools of algebraic geometry (effective Lang-Weil estimates
and intrinsic degree bounds for annihilating polynomials of
circuits) with information-theoretic arguments (submodularity of
entropy and conditional entropy estimates).

These results provide the first effective and uniform approximation bounds
for algebraic matroids by entropy polymatroids, clarifying the quantitative link
between algebraic independence (captured by matroid rank) and information-theoretic
independence (captured by entropy).
\end{abstract}
\maketitle

\section{Introduction}

Matroid theory provides a unifying framework for combinatorial
structures arising in graph theory, geometry, and optimization.
Among the various classes of matroids, \emph{algebraic matroids}
occupy a distinguished position. They arise naturally from sets of
algebraically dependent elements in a field extension and serve as
a bridge between combinatorial independence and algebraic
geometry. Classical references such as \cite{Oxley11} (see also
\cite{RoSiTh20}) emphasize their central role both in combinatorics and in
applications to algebraic geometry.

An important line of research initiated by
Mat\'u$\check{\textrm{s}}$ \cite{Matus99,Matus24} (see also
\cite{BenEfraim16}) concerns the relationship between algebraic
matroids and \emph{entropic matroids}. In particular,
Mat\'u$\check{\textrm{s}}$ showed that every algebraic matroid is
\emph{almost entropic}, meaning that its rank function can be
approximated by entropic polymatroids. This result establishes a
deep connection between information theory and algebraic
combinatorics. While the original proof is qualitative in nature, it
leaves open the problem of obtaining explicit quantitative bounds—
a question that is the focus of the present work.

A quantitative refinement of Mat\'u$\check{\textrm{s}}$'s theorem is
important not only from a conceptual perspective, but also for
potential applications. Explicit bounds on the rate at which
entropic polymatroids approximate the rank function of an algebraic
matroid are relevant, for instance, in the design of
information-theoretic protocols such as polynomial secret sharing
schemes \cite{Shamir79, Beimel11}, where sequences of entropic
inequalities allow one to construct secret sharing schemes with
information ratio tending to one (see \cite{BeFaMo25} for details).
At the same time, quantitative approximations play a key role in
network coding \cite{Yeung08}, where entropic polymatroids provide
outer bounds for feasible rate regions.

The purpose of this paper is to present a refined and quantitative
version of Mat\'u$\check{\textrm{s}}$'s theorem. As shown
in \cite{Matus24}, it suffices to consider the case of algebraic
matroids over a finite field. Our main
contribution is an explicit approximation scheme showing that
algebraic matroids over a finite field are almost entropic, with the error
bounded in terms of geometric invariants of the associated algebraic
variety. More precisely, we obtain the following result.
\begin{theorem}\label{th: main intro}
Let ${\sf M}:=(\mathcal{E},r)$ be an algebraic matroid defined
over a finite field $\G$, where $\mathcal{E}$ is its ground set and $r$ is the
corresponding rank function. Let $V\subset\A^n_{\overline{\G}}$ be
the variety associated with ${\sf M}$, and let $\F$ be a
finite field extension of $\G$ with $|\F|\ge (5\deg(V)^{7/3}+1)$.
For an arbitrary
subset $\mathcal{K}\subseteq \mathcal{E}$, with $n:=|\mathcal{E}|$,
$m:=r(\mathcal{E})$, $\delta:=\deg(V)$ and $q:=|\F|$, we have
$$|h_\F(\mathcal{K}) - r(\mathcal{K})|\le
\frac{nm\delta^2}{q^{1/2}}+\frac{n\ln\delta}{\ln q}
-\frac{\ln(1-{5\delta^{7/3}}{q^{-1/2}})}{\ln q}.
$$
\end{theorem}
This yields a quantitative refinement of Mat\'u$\check{\textrm{s}}$'s theorem,
with effective constants associated to the underlying variety $V$.

Beyond the proof of this quantitative version, we also undertake a
systematic study of \emph{theoretical aspects} of the varieties
associated with algebraic matroids. In this part of the paper we
analyze some of their algebraic, geometric, and arithmetic
properties, and explain how these features influence the entropic
behavior of the corresponding matroids. In particular, we highlight
the role of dimension theory, degrees, and rational points in
controlling the quality of entropic approximations. As an example,
we obtain the following result.
\begin{proposition}\label{prop: degree bound f_C intro}
With hypotheses and notations as in Theorem \ref{th: main intro},
let $\mathcal{C}\subset\mathcal{E}$ be a circuit and let
$j\in\mathcal{C}$. Then there exists a nonzero absolutely
irreducible polynomial $f_\mathcal{C} \in \G[X_i : i \in
\mathcal{C}]\setminus\G[X_i : i \in \mathcal{C}\setminus\{j\}]$ of
degree at most $\delta$ such that
\[
f_\mathcal{C}((f_i)_{i \in \mathcal{C}}) = 0.
\]
\end{proposition}

The existence of an annihilating polynomial $f_\mathcal{C}$ for
the set $\{f_i:i\in\mathcal{C}\}$ is well-known (see, e.g., \cite[Theorem 11]{RoSiTh20}).
Furthermore, an {\em extrinsic} upper bound on the
degree of this polynomial in terms of the degrees of the polynomials
$\{f_i:i\in\mathcal{C}\}$ can be obtained applying \cite[Corollary
6]{BeMiSa13}. Our contribution relies on the {\em intrinsic} bound
of the proposition in terms of the degree of the variety $V$.

The paper is organized as follows. In Section~\ref{section:
preliminaries} we review the necessary background on information
theory, matroid theory, and algebraic geometry.
Section~\ref{section: geometry of algebraic matroids} discusses
theoretical aspects of algebraic, geometric, and arithmetic nature
of the variety underlying an algebraic matroid. Finally,
Section~\ref{section: main} contains the precise statement and proof
of our main theorem. 

\section{Preliminaries}
\label{section: preliminaries}
In this section we recall basic notions on information theory,
matroid theory, and algebraic geometry that will be used throughout
the paper.
%
%
\subsection{Entropy and polymatroids}

Given a discrete random variable $S$ taking values in a finite set
$\mathcal{S}$ with probability mass function
$p(s) := \mathbb{P}(S=s)$, the \emph{Shannon entropy} of $S$ is defined by
\[
H(S):= - \sum_{s \in \mathcal{S}} p(s) \ln p(s),
\]
where $\ln$ represents the logarithm in base $e$, and where terms
with $p(s)=0$ are omitted. For two random variables $S$ and $T$ with
joint probability distribution $p(s,t)$, one defines the \emph{joint
entropy}
\[
H(S,T):= - \sum_{s \in \mathcal{S}} \sum_{t \in \mathcal{T}}
p(s,t)\log p(s,t),
\]
and the \emph{conditional entropy} of $S$ given $T$ as
\[
H(S \mid T):= \sum_{t \in \mathcal{T}} p(t) H(S \mid T=t),
\]
where
\[
H(S \mid T = t) := - \sum_{s \in \mathcal{S}} p(s \mid t) \, \ln
p(s \mid t).
\]
Equivalently,
\[
H(S \mid T):= H(S,T) - H(T).
\]

These notions extend to random vectors $S:=(S_i)_{i\in\mathcal{E}}$
indexed by a finite set $\mathcal{E}$, where for $\mathcal{A}\subset
\mathcal{E}$ one considers the subvector
$S_\mathcal{A}:=(S_i)_{i\in\mathcal{A}}$. Associated to such a
vector is the set function $h: 2^\mathcal{E}\to\mathbb{R}_{\ge 0}$,
$h(\mathcal{A}) := H(S_\mathcal{A})$. A function $h:
2^\mathcal{E}\to\mathbb{R}_{\ge 0}$ is called a \emph{polymatroid}
if it is:
\begin{itemize}
  \item normalized ($h(\emptyset)=0$),
  \item monotone ($\mathcal{A}\subset\mathcal{B}$ implies
$h(\mathcal{A})\le h(\mathcal{B})$), and
  \item submodular
($h(\mathcal{A})+h(\mathcal{B}) \ge h(\mathcal{A}\cup \mathcal{B})+h(\mathcal{A}\cap \mathcal{B})$
for all $\mathcal{A},\mathcal{B}$).
\end{itemize}
When $h$ arises as above from
the entropies of a random vector $S$, it is said to be
\emph{entropic}. More precisely, $h$ is called the
entropic polymatroid \emph{induced} by the
random vector $S$.

\subsection{Matroids and algebraic matroids}
\label{subsec: matroids}

We next recall basic notions of matroid theory; see \cite{Oxley11}
for background.

Let $\mathcal{E}$ be a finite set. A \emph{matroid} on $\mathcal{E}$
is a pair ${\sf M}:= (\mathcal{E}, r)$ where $r : 2^\mathcal{E} \to
\mathbb{Z}_{\geq 0}$ is a function, called the \emph{rank function},
satisfying the following axioms for all subsets $\mathcal{A},
\mathcal{B} \subset \mathcal{E}$:
\begin{enumerate}
  \item[\textnormal{(R1)}] $0 \leq r(\mathcal{A}) \leq |\mathcal{A}|$,
  \item[\textnormal{(R2)}] If $\mathcal{A} \subset\mathcal{B} \subset \mathcal{E}$, then $r(\mathcal{A}) \leq r(\mathcal{B})$,
  \item[\textnormal{(R3)}] For all $\mathcal{A}, \mathcal{B} \subset \mathcal{E}$,
  \[
  r(\mathcal{A} \cup \mathcal{B}) + r(\mathcal{A} \cap \mathcal{B}) \leq r(\mathcal{A}) + r(\mathcal{B}).
  \]
\end{enumerate}
These axioms imply that $r$ is a normalized, monotone,
and submodular function with integer values.

A subset $\mathcal{I}\subset \mathcal{E}$ is said to be
\emph{independent} if $r(\mathcal{I})=|\mathcal{I}|$, and a
\emph{base} is an inclusion-maximal independent set. A
\emph{circuit} is a minimally dependent subset, i.e., a dependent
set all of whose proper subsets are independent. If $\mathcal{I}$ is
a base and $j\in \mathcal{E}\setminus\mathcal{I}$, then the
\emph{fundamental circuit of $j$ with respect to $\mathcal{I}$} is
the unique circuit contained in $\mathcal{I}\cup\{j\}$.

Using the notion of entropic polymatroid introduced above,
we call a matroid ${\sf M}:=(\mathcal{E},r)$
\emph{almost entropic} if its rank function $r$ lies in the closure
of the cone of entropic polymatroids, i.e., if there exists a
sequence of entropic polymatroids $(h^{(m)})_{m\in\mathbb{N}}$
such that for every subset $\mathcal{A} \subset \mathcal{E}$,
\[
\lim_{m \to \infty} h^{(m)}(\mathcal{A}) = r(\mathcal{A}).
\]

A matroid ${\sf M}:=(\mathcal{E},r)$ is said to be \emph{algebraic
over a field $\G$} if there exists a family of elements
$(f_i)_{i\in\mathcal{E}}$ in a purely transcendental extension
$\mathbb{H}$ of $\G$ such that for every $\mathcal{A}\subset
\mathcal{E}$ the rank is given by
\[
r(\mathcal{A}) = \deg_{\operatorname{tr}/\G}(f_i : i \in \mathcal{A}),
\]
where $\deg_{\operatorname{tr}/\G}$ denotes
the transcendence degree of $\G(f_i : i \in \mathcal{A})$ over $\G$.

%
%
\subsection{Algebraic geometry}

We use standard notions and notations of
algebraic geometry, as found, for example, in \cite{Harris92},
\cite{Kunz85}, and \cite{Shafarevich94}.

Let $\G$ be a perfect field and $\overline{\G}$ its algebraic
closure. We denote by $\A_{\overline{\G}}^n$ the $n$--dimensional
affine space $\overline{\G}{}^{n}$. This space is endowed with its
Zariski topologies over $\G$, for which a closed set is the zero
locus of a set of polynomials of $\G[X_1,\ldots, X_{n}]$.
An {\em affine variety of $\A_{\overline{\G}}^n$ defined over} $\G$
(or an affine $\G$--variety) is the set of common zeros in
$\A_{\overline{\G}}^n$ of polynomials $F_1,\ldots, F_{m} \in
\G[X_1,\ldots, X_{n}]$. 
We will often write $V(F_1\klk F_m)$ or $\{F_1=0\klk F_m=0\}$ to
denote this affine $\G$--variety. 

A $\G$--variety $V$ is said to be $\G$--{\em irreducible} if it
cannot be expressed as a finite union of proper $\G$--subvarieties
of $V$. Furthermore, $V$ is {\em absolutely irreducible} if it is
$\overline{\G}$--irreducible as a $\overline{\G}$--variety. Any
$\G$--variety $V$ admits a unique (up to reordering) irredundant
decomposition $V=\mathcal{C}_1\cup \cdots\cup\mathcal{C}_s$, where
each $\mathcal{C}_i$ is  $\G$--irreducible  (absolutely irreducible
and defined over $\G$). These are called the {\em irreducible}
$\G$--{\em components} ({\em absolutely irreducible components}) of
$V$.

For a $\G$--variety $V$ of $\A_{\overline{\G}}^n$, we denote by
$I(V)$ its {\em defining ideal}, namely, the set of all polynomials
in $\G[X_1,\ldots, X_n]$ that vanish on $V$. The {\em coordinate
ring} $\G[V]$ of $V$ is defined as $\G[X_1,\ldots,X_n]/I(V)$. When
$V$ is irreducible, then $\G[V]$ is a domain, and its field of
fractions $\G(V)$ is called the {\em field of rational functions} of
$V$. The {\em dimension} $\dim V$ of $V$ is the maximal length $r$
of chains $V_0\varsubsetneq V_1 \varsubsetneq\cdots \varsubsetneq
V_r$ of nonempty irreducible $\G$--varieties contained in $V$. We
say that $V$ has {\em pure dimension} $r$, or it is {\em
equidimensional of dimension $r$}, if all its irreducible
$\G$--components are of dimension $r$.

The {\em degree} $\deg V$ of an equidimensional $\G$--variety $V$ is
the maximum number of points in the intersection $V\cap L$, where
$L$ is a linear space of codimension $\dim V$ such that $V\cap L$ is
a finite set. More generally, following \cite{Heintz83} (see also
\cite{Fulton84}), if $V=\mathcal{C}_1\cup\cdots\cup \mathcal{C}_s$
is the decomposition of $V$ into irreducible $\G$--components, we
define the degree of $V$ as
$$\deg V:=\sum_{i=1}^s\deg \mathcal{C}_i.$$
We will frequently use the following {\em B\'ezout inequality} (see
\cite{Heintz83}, \cite{Fulton84}): if $V$ and $W$ are
$\G$--varieties of $\A_{\overline{\G}}^n$, then
\begin{equation}\label{eq: Bezout}
\deg (V\cap W)\le \deg V \cdot \deg W.
\end{equation}

%
%

Let $V$ and $W$ be irreducible affine $\G$--varieties of the same
dimension, and let $f:V\to W$ be a regular map for which
$\overline{f(V)}=W$, where $\overline{f(V)}$ is the closure of
$f(V)$ with respect to the Zariski topology of $W$. Such a map is
called {\em dominant}. Then $f$ induces a ring extension
$\G[W]\hookrightarrow \G[V]$ by composition with $f$.

%
\subsection{Rational points}
Let $\fq$ be the finite field of $q$ elements and let $\cfq$ be the
algebraic closure of $\fq$. Let $\A^n(\fq)$ be the $n$--dimensional
$\fq$--vector space $\fq^n$. For an affine variety
$V\subset\A_{\cfq}^n$, we denote by $V(\fq)$ the set of
$\fq$--rational points of $V$, namely $V(\fq):=V\cap \A^n(\fq)$.
%
%
If $V$ is an affine variety of dimension $m$ and degree $\delta$,
then (see, e.g., \cite[Lemma 2.1]{CaMa06})
 \begin{equation}\label{eq: upper bound -- affine gral}
   |V(\fq)|\leq \delta q^m.
 \end{equation}


\section{On the algebraic and geometric structure of algebraic matroids}
\label{section: geometry of algebraic matroids}
Let $\G$ be a finite field and let ${\sf M}:= (\mathcal{E}, r)$ be
an algebraic matroid over $\G$, that is, there is a family of elements
$(f_i)_{i \in \mathcal{E}}$ in a purely transcendental field
extension $\mathbb{H}:=\G(X_1,\ldots,X_s)$ of $\G$ such that the
rank function $r:=2^\mathcal{E} \to\mathbb{Z}_{\ge 0}$ satisfies
\[
r(\mathcal{I}) := \deg_{\operatorname{tr}/\G}(f_i : i \in
\mathcal{I}),
\]
where $\deg_{\operatorname{tr}/\G}$ denotes the transcendence degree
of the field extension $\G(f_i : i \in \mathcal{I})$ over $\G$.
We denote $n:=|\mathcal{E}|$ and let $d$ be the maximum of the degrees of
the polynomials $f_1,\ldots,f_n$. We shall assume that
$m:=r(\mathcal{E})<n$.

In this section we establish a number of facts concerning algebraic,
geometric, and arithmetic aspects of the algebraic matroid
${\sf M}:= (\mathcal{E}, r)$.

Let $I_\G$
be the ideal of $\G[X_1,\ldots,X_n]$ consisting of all polynomials
$f$ such that $f(f_1,\ldots,f_n) = 0$.
\begin{remark}
$I_\G$ is a prime ideal of $\G[X_1,\ldots,X_n]$.
\end{remark}
\begin{proof}
Let $\Phi:\A_{\overline{\G}}^n\to\A_{\overline{\G}}^n$ the
polynomial map defined by $\Phi(\bfs x):=(f_1(\bfs
x),\ldots,f_n(\bfs x))$. Then $f$ belongs to $I_\G$ is and only if
$f$ vanishes on the image $\mathrm{Im}(\Phi)$ of $\Phi$.
Equivalently, $f$ vanishes on the Zariki closure of
$\mathrm{Im}(\Phi)$. As $\A_{\overline{\G}}^n$ is irreducible, the
Zariski closure of $\mathrm{Im}(\Phi)$ is also irreducible, and
therefore its ideal $I_\G$ is prime.
\end{proof}

It follows that $V:= V(I_\G)\subset\A_{\overline{\G}}^n$ is an
affine absolutely irreducible variety defined over $\G$. The
dimension of $V$ equals the transcendence degree of its function
field, which in this case is the field of fractions of
$\G[X_1,\ldots,X_n]/I_\G\cong\G[f_1,\ldots,f_n]$. As a consequence,
we obtain the following remark.
\begin{remark}\label{remark: dimension}
The dimension of $V$ equals $\deg_{\operatorname{tr}/\G}(f_i : i \in
\mathcal{E})=r(\mathcal{E})$.
\end{remark}

Although we shall express our estimates in terms of geometric
invariants associated to the variety $V$, such as its dimension and
degree, we first provide a simple upper bound on the degree of
$V$ in terms of the degrees of $f_1,\ldots,f_n$.
\begin{lemma}\label{lemma: degree bound V_G}
If $m:=r(\mathcal{E})$ denotes the dimension of $V$, then $\deg
V\le d^m$.
\end{lemma}
\begin{proof}
As stated before, if $\Phi:=(f_1\klk
f_n):\A_{\overline{\G}}^n\to\A_{\overline{\G}}^n$ is the map defined
above, then $V\subset\A_{\overline{\G}}^n$ is the Zariski closure of
$\mathrm{Im}(\Phi)$. Let $H_1\klk H_m$ be hyperplanes of
$\A_{\overline{\G}}^n$ such that $|V\cap H_1\cap\cdots\cap H_m|=\deg
V$. Since $\mathrm{Im}(\Phi)$ is a dense subset of $V$ containing a
Zariski-open dense subset, we may assume without loss of generality
that $V\cap H_1\cap\cdots\cap H_m=\mathrm{Im}(\Phi)\cap
H_1\cap\cdots\cap H_m$. Let $\mathcal{S}:=\mathrm{Im}(\Phi)\cap
H_1\cap\cdots\cap H_m$. Then
$$\Phi^{-1}(\mathcal{S})=\Phi^{-1}(H_1)\cap\cdots\cap \Phi^{-1}(H_m).$$

Observe that $\Phi^{-1}(H_i)=\{g_i=0\}$, where $g_i$ is a linear
combination of the polynomials $f_1\klk f_n$ for $i=1\klk m$.
Therefore, by the B\'ezout inequality \eqref{eq: Bezout} it follows
that
$$\deg \Phi^{-1}(\mathcal{S})\le \deg(\Phi^{-1}(H_1))\cdots\deg(\Phi^{-1}(H_m))\le
\deg g_1\cdots\deg g_m\le  d^m.$$
Let $\Phi^{-1}(\mathcal{S})=\cup_{i=1}^N\mathcal{C}_i$ be the
decomposition of $\Phi^{-1}(\mathcal{S})$ into irreducible
components. Since $\Phi(\Phi^{-1}(\mathcal{S}))=\mathcal{S}$ and
each irreducible component $\mathcal{C}_i$ of
$\Phi^{-1}(\mathcal{S})$ is mapped by $\Phi$ to a point of
$\mathcal{S}$, we have
$$\deg V=|\mathcal{S}|\le N\le\sum_{i=1}^N\deg \mathcal{C}_i=\deg
\Phi^{-1}(\mathcal{S})\le d^m.$$
This completes the proof of the lemma.\end{proof}

For any finite field extension $\F$ of $\G$, let $V(\F):= V \cap
\F^n$ denote the set of $\F$-rational points of $V$.
As a consequence of \cite[Theorem 5.4]{CaMa06}, we have the
following result.
\begin{lemma}\label{lemma: condition V_F is nonempty}
Let $\delta:=\deg V$. If $|\F|\ge 2\delta^4$, then the set $V(\F)$ is nonempty.
\end{lemma}

In the following, we will need a lower bound on the number $|V(\F)|$ of
$\F$-rational points of $V(\F)$.
\begin{lemma}\label{lemma: explicit Lang-Weil}
Let $q:=|\F|>(5\delta^{7/3}+1)^2$. Then
$$|V(\F)|\ge q^m-5\delta^{7/3}q^{m-1/2}>q^{m-1/2}.$$
\end{lemma}
\begin{proof}
According to \cite[Theorem 7.1]{CaMa06},
for the absolutely irreducible variety $V$ defined over $\F$ we have
the following estimate:
$$||V(\F)|-q^m|\le \delta^2q^{m-1/2}+5\delta^{13/3}q^{m-1}.$$
Taking into account that $q:=|\F|\ge 2(\deg V)^4$, we obtain
\begin{align*}
|V(\F)|&\ge q^m-\delta^2q^{m-1/2}-5\delta^{13/3}q^{m-1}\notag\\
&\ge q^m-\big(\delta^2+\mbox{$\frac{5}{\sqrt{2}}$}\delta^{7/3}\big)q^{m-1/2}
\ge q^m-5\delta^{7/3}q^{m-1/2}.
\end{align*}
This proves the lemma.
\end{proof}


We conclude this section with a discussion of circuits in
$\mathcal{E}$. To any circuit $\mathcal{C}$ of $\mathcal{E}$ we can
naturally associate an annihilating polynomial $f_\mathcal{C}$ of
the set $\{f_i:i\in\mathcal{C}\}$ (see, e.g., \cite[Theorem 11]{RoSiTh20}).
An {\em extrinsic} upper bound on the
degree of this polynomial in terms of the degrees of the polynomials
$\{f_i:i\in\mathcal{C}\}$ can be obtained from \cite[Corollary
6]{BeMiSa13}. Nevertheless, we are interested in {\em intrinsic} bounds
in terms of invariants associated with the variety $V$. In this direction,
we have the following result.
\begin{proposition}\label{prop: degree bound f_C}
Let $\mathcal{C}\subset\mathcal{E}$ be a circuit and let
$j\in\mathcal{C}$. Then there exists a nonzero absolutely
irreducible polynomial $f_\mathcal{C} \in \G[X_i : i \in
\mathcal{C}]\setminus\G[X_i : i \in \mathcal{C}\setminus\{j\}]$ of
degree at most $\delta$ such that
\[
f_\mathcal{C}((f_i)_{i \in \mathcal{C}}) = 0.
\]
\end{proposition}
\begin{proof}
By definition, the polynomials $(f_i)_{i\in\mathcal{C}}$ are
algebraically dependent. Therefore, there exists a nonzero
polynomial $f_\mathcal{C} \in \overline{\G}[X_i : i \in
\mathcal{C}]$ such that $f_\mathcal{C}((f_i)_{i \in \mathcal{C}}) =
0$. Without loss of generality we may take $f_\mathcal{C}$ in
$\G[X_i : i \in \mathcal{C}]$, and irreducible in $\overline{\G}[X_i
: i \in \mathcal{C}]$. As $\mathcal{C}$ is a circuit, the set
$\{f_i:i\in\mathcal{C}\setminus\{j\}\}$ is algebraically
independent. It follows that $f_\mathcal{C}$ cannot be a polynomial
in $\G[X_i : i \in \mathcal{C}\setminus\{j\}]$.

To prove the degree bound for $f_\mathcal{C}$, note that since the
set $\mathcal{C}\setminus\{j\}$ is independent, we may extend it to
a base $\mathcal{I}\subset\mathcal{E}$ of ${\sf M}$. We have
$|\mathcal{I}|=m=r(\mathcal{E})=\dim V$, and
$\{X_i:i\in\mathcal{I}\}$ is a transcendence basis of the field
extension $\G\hookrightarrow\G(V)$. In particular, we have the field
extension $\G(X_i:i\in\mathcal{I})\hookrightarrow\G(V)$. This
implies that the morphism $\pi_\mathcal{I}:V\to\A^m_{\overline{\G}}$
defined by $\pi_\mathcal{I}(\bfs x):=(x_i:i\in\mathcal{I})$ is
dominant (see, e.g., \cite[\S 3.3, Exercise 6]{Kunz85}). Applying
\cite[Proposition 1]{Schost03}, we conclude that there exists a
nonzero polynomial
$m_{j,\mathcal{I}}\in\G[X_i:i\in\mathcal{I}\cup\{j\}]\setminus
\G[X_i:i\in\mathcal{I}]$ of minimal degree with $\deg
m_{j,\mathcal{I}}\le \delta$ such that
$$m_{j,\mathcal{I}}(X_i:i\in\mathcal{I}\cup\{j\})=0\quad \textrm{in }\G[V].$$
More precisely, the ideal $I_\G\cap\G[X_i:i\in\mathcal{I}\cup\{j\}]$
is principal and $m_{j,\mathcal{I}}$ is a generator of this ideal.
As $f_\mathcal{C}$ also belongs to
$I_\G\cap\G[X_i:i\in\mathcal{I}\cup\{j\}]$, we conclude that
$m_{j,\mathcal{I}}$ divides $f_\mathcal{C}$ in
$\G[X_i:i\in\mathcal{I}\cup\{j\}]$. As
$f_\mathcal{C}$ also belongs to $\G[X_i:i\in\mathcal{C}\cup\{j\}]$ and it
is irreducible, we conclude that $m_{j,\mathcal{I}}$ also belongs to
$\G[X_i:i\in\mathcal{C}\cup\{j\}]$ and differs from $f_\mathcal{C}$
by a nonzero constant in $\G$. This proves the degree bound.
%
\end{proof}

\begin{example}[Gap between extrinsic bounds and
the intrinsic bound of Proposition \ref{prop: degree bound f_C}]
Let $\G$ be a perfect field, let $n\ge 2$ and consider the polynomial morphism
\[
\varphi:\A^n_{\overline{\G}} \longrightarrow \A^{n+1}_{\overline{\G}},\qquad
\bfs x:=(x_1,\dots,x_n) \longmapsto \big(x_1,\sigma_1(\bfs x),\ldots,\sigma_n(\bfs x)\big),
\]
where $\sigma_j$ denotes the $j$th elementary symmetric polynomial
in $n$ variables.

Let $V:=\overline{\varphi(\A^n_{\overline{\G}})}\subset \A^{n+1}_{\overline{\G}}$
denote the Zariski closure of the image of $\varphi$.
Consider the algebraic matroid ${\sf M}$ on the ground set $\mathcal{E}:=\{1,\dots,n+1\}$
defined by the algebraic dependencies among the coordinate functions of
$\varphi$.

The associated variety of ${\sf M}$ is $V$, which
is the hypersurface defined by the classical relation
$$x_1^n-\sigma_1(\bfs x)x_1^{n-1}+\cdots+(-1)^n\sigma_n(\bfs x)=0,$$
of degree $\deg(V) =n$. On the other hand,
the extrinsic bound of \cite[Corollary 6]{BeMiSa13} applied to the parametrization
$\varphi$ gives $\deg(V)\le n^n$, since the map $\varphi$ is defined by $n+1$ polynomials
of degree at most $n$, and the hypersurface $V$ has dimension $n$.
From the perspective of our entropy bounds, this example is particularly
illustrative: the actual degree of the annnihilating polynomial of the only circuit
of the matroid, namely $\{1,\ldots,n+1\}$, is $n$, while the extrinsic (optimal in general) bound
of \cite[Corollary 6]{BeMiSa13} grows exponentially in $n$.

This example shows that our entropy estimates, expressed intrinsically in terms
of $\deg(V)$, can be considerably sharper than those derived from extrinsic degree bounds.
\end{example}

In summary, this section established the algebraic, geometric, and
arithmetic framework underlying algebraic matroids. We derived explicit
degree bounds (Lem\-ma~\ref{lemma: degree bound V_G}) for
the associated variety $V$ and obtained
estimates concerning its number of rational points over a finite field
(Lemmas~\ref{lemma: condition V_F is nonempty} and
\ref{lemma: explicit Lang-Weil}). Moreover, we analyzed circuits and
proved that their algebraic dependencies admit intrinsic degree bounds
in terms of the degree of $V$ (Proposition~\ref{prop: degree bound f_C}).
These results provide the structural foundation for the entropic
approximation of algebraic matroids developed in the next section.


\section{Quantitative entropic approximation of algebraic matroids}
\label{section: main}
The connection between algebraic matroids and entropic polymatroids
is of particular importance in information theory. Entropic
polymatroids arise naturally as entropy functions of collections of
random variables, and thus capture fundamental constraints in
communication, coding, and secret sharing. Showing that algebraic
matroids are almost entropic establishes that combinatorial
independence relations of algebraic origin can be asymptotically
realized by information-theoretic models. This perspective provides
a bridge between algebraic geometry and information theory, and
motivates the study of quantitative refinements, where one seeks
explicit error bounds in the approximation of rank functions by
entropic polymatroids.

Let ${\sf M}:= (\mathcal{E}, r)$ be an algebraic matroid over a
finite field $\G$ as in Section \ref{section: geometry of algebraic
matroids}, i.e., there exist elements $(f_i)_{i \in \mathcal{E}}$ in
a purely transcendental field extension $\mathbb{H}$ of $\G$ such
that the rank function $r$ is given by
\[
r(\mathcal{I}) := \deg_{\text{tr}/\G}(f_i : i \in \mathcal{I})
\]
for all subsets $\mathcal{I} \subset \mathcal{E}$.

A construction due to~\cite{Matus24} (see also \cite{BeFaMo25}) shows
that algebraic matroids are almost entropic. More precisely, the
author exhibits a family of entropic polymatroids whose entropies
approximate the rank function $r$ of ${\sf M}$. However,
several open questions remain concerning this proof, which we will
address in the following.

We begin with a description of the construction of \cite{Matus24}.

%
\subsection{An entropic polymatroid}
Let $V\subset\A^n_{\overline{\G}}$ be the variety associated with
the algebraic matroid ${\sf M}:=(\mathcal{E},r)$ as in Section
\ref{section: geometry of algebraic matroids}. Let $\F$ be a finite
field extension of $\G$ such that the set of $\F$-rational points
$V(\F)$ of $V$ is nonempty (for example, satisfying the hypothesis
of Lemma \ref{lemma: condition V_F is nonempty}). Define a random
variable $S:= (S_i)_{i \in \mathcal{E}} \in \F^\mathcal{E}$ taking
values in $V(\F)$, where each $\bfs z \in V(\F)$ is chosen with
equal probability, that is,
\[
\mathbb{P}(S=\bfs z) = \frac{1}{|V(\F)|}, \quad \text{for all } \bfs
z \in V(\F).
\]
For each $\mathcal{I} \subset \mathcal{E}$, let $S_\mathcal{I} :=
(S_i)_{i \in \mathcal{I}}$ denote the projection of $S$ onto the
coordinates indexed by $\mathcal{I}$. This defines a probability
distribution $\mathbb{P}_\mathcal{I}$ on $\F^\mathcal{I}$. Define a
function $h_\F: 2^\mathcal{E} \to \mathbb{R}_{\geq 0}$ by
\begin{equation}\label{eq: definition polymatroid}
h_\F(\mathcal{I}) := \frac{H(S_\mathcal{I})}{\ln |\F|},
\end{equation}
where $H(S_\mathcal{I})$ denotes the Shannon entropy of the random
variable $S_\mathcal{I}$, and $\ln$ denotes the logarithm to the
base $e$. The function $h_\F$ is normalized, monotone, and
submodular; hence it is a polymatroid. Since it is defined via
entropies of actual random variables, $h_\F$ is an entropic
polymatroid.

It is shown in~\cite{Matus24} (see also \cite{BeFaMo25}) that, as
$|\F| \to \infty$, the function $h_\F$ converges pointwise to $r$,
in the sense that for all $\mathcal{A} \subset \mathcal{E}$,
\begin{equation}\label{eq: convergence polymatroid}
\left| h_\F(\mathcal{A}) - r(\mathcal{A}) \right| \leq
\frac{\mathfrak{n}}{\ln |\F|},
\end{equation}
for some constant $\mathfrak{n}$ independent of $q:=|\F|$. In
particular, the matroid ${\sf M}$ is almost entropic: its rank
function is the pointwise limit of entropic polymatroids.

Nevertheless, several questions remain open in this
proof. First, the size of the field $\F$ such that the set $V(\F)$
is nonempty is not known, and this size defines later the first field
where the sequence of polymatroids is defined.

Moreover, an explicit constant $\mathfrak{n}$ in \eqref{eq:
convergence polymatroid} would provide a more precise understanding of the
approximation of the rank function of the polymatroids $h_\F$, and
any improvement of this bound would directly sharpen the rate of convergence of
the polymatroids to the original matroid.

Our purpose is to address these issues, providing explicit
expressions for the constants underlying this construction in terms
of geometric invariants associated to the variety $V$, which can be
effectively computed from the polynomials $(f_i:i\in\mathcal{E})$.


\subsection{Explicit bounds on the convergence rate}

From now on, let $V\subset\A^n_{\overline{\G}}$ be the variety
associated to the matroid ${\sf M}:=(\mathcal{E},r)$ under consideration,
and let $\F$ be a finite field extending $\G$ and satisfying the hypothesis of
Lemma \ref{lemma: explicit Lang-Weil}. Denote by $\delta:=\deg(V)$. We shall prove that the
polymatroids defined in \eqref{eq: definition polymatroid} converge
to the algebraic matroid ${\sf M}$ under consideration when the
cardinality $q:=|\F|$ tends to infinity. Furthermore, we
provide explicit estimates for the convergence rate.

We start with the analysis of independent sets
$\mathcal{I}\subset\mathcal{E}$.


\subsubsection{Independent sets}
Let $\mathcal{I}$ be an arbitrary subset of $\mathcal{E}$. For $\bfs
y_\mathcal{I}\in\pi_\mathcal{I}(V(\F))\subset\F^\mathcal{I}$, denote
by
$$\pi^{-1}_\mathcal{I}(\bfs y_\mathcal{I}):=\{\bfs x\in V(\F):
\pi_\mathcal{I}(\bfs x)=\bfs y_\mathcal{I}\}$$
its $\F$-fiber under
$\pi_\mathcal{I}$. The probability that
$\mathcal{S}_\mathcal{I}=\bfs y_\mathcal{I}$ is
$$\mathbb{P}_\mathcal{I}(\mathcal{S}_\mathcal{I}=\bfs y_\mathcal{I})=
\frac{|\pi^{-1}_\mathcal{I}(\bfs y_\mathcal{I})|}{|V(\F)|}.$$
As a consequence, the entropy of the marginal
probability $\mathbb{P}_\mathcal{I}$ can be written as
\begin{equation}\label{eq: entropy P_I}
H(\mathbb{P}_\mathcal{I})=
-\sum_{\bfs y_\mathcal{I}\in\pi_\mathcal{I}(V(\F))}
\frac{|\pi_\mathcal{I}^{-1}(\bfs y_\mathcal{I})|}{|V(\F)|}\ln
\frac{|\pi_\mathcal{I}^{-1}(\bfs y_\mathcal{I})|}{|V(\F)|}.
\end{equation}

Following the approach of \cite{Matus24}, we first consider the case of bases,
and then extend the estimates to arbitrary independent sets.
Let $\mathcal{I}\subset\mathcal{E}$ be a base of ${\sf M}$,
i.e., a maximal independent subset of $\mathcal{E}$. Let $j\in
\mathcal{E}\setminus \mathcal{I}$ and $\mathcal{C}_j:=
\gamma(j,\mathcal{I})$ denote the fundamental circuit of $j$ with
respect to $\mathcal{I}$. Let
$f_{\mathcal{C}_j}\in\G[X_i:i\in\mathcal{C}]\setminus\G[X_i:i\in\mathcal{I}]$
be the absolutely irreducible polynomial whose existence is guaranteed
by Proposition \ref{prop: degree bound f_C}.
In order to estimate the entropy expression in \eqref{eq: entropy P_I},
we split the sum distinguishing those $\bfs y_\mathcal{I}$
for which there exists $j\in\mathcal{E}\setminus
\mathcal{I}$ such that the specialization $f_{\mathcal{C}_j}(\bfs
y_\mathcal{I},X_j)$ vanishes identically.

Define
$$
W_j:=\{\bfs x\in V:f_{\mathcal{C}_j}(\pi_\mathcal{I}(\bfs x),X_j)\textrm{ is the zero polynomial}\},\quad
W:=\bigcup_{j\in\mathcal{E}\setminus\mathcal{I}}W_j.
$$
Let $\bfs y_\mathcal{I}\in \pi_\mathcal{I}(W(\F))$ and let
$j\in\mathcal{E} \setminus\mathcal{I}$ be such that
$f_{\mathcal{C}_j}(\bfs y_\mathcal{I},X_j)$ is identically zero.
Write $f_{\mathcal{C}_j}$ as a polynomial in
$\G[X_i:i\in\mathcal{I}][X_j]$:
$$f_{\mathcal{C}_j}=\sum_{h\in\mathcal{H}_j}m_hX_j^h,$$
where each $m_h\in\G[X_i:i\in\mathcal{I}]$ is nonzero. Let $h_j$ be
the largest element of $\mathcal{H}_j$. Then $W_j\subset
\{m_{h_j}=0\}\cap V$. Since $\mathcal{I}$ is an independent set, we
have $I_\G\cap \G[X_i:i\in\mathcal{I}]=\{0\}$. Thus $m_{h_j}$ is
neither a zero divisor nor zero in the quotient ring
$\G[X_1,\ldots,X_n]/I_\G$, and therefore $V\cap\{m_{h_j}=0\}$ is a
subvariety of pure codimension 1 of $V$ (see, e.g., \cite[Chapter V,
Corollary 3.2]{Kunz85}). It follows that
$$|W_{\mathcal{C}_j}(\F)|\le |V\cap\{m_{h_j}=0\}\cap\F^n|\le
\deg(V\cap\{m_{h_j}=0\})q^{\dim V-1}\le \delta^2\, q^{m-1},$$
where the second inequality follows from \eqref{eq: upper bound -- affine gral},
and the third from the B\'ezout inequality \eqref{eq: Bezout}.
Hence,
$$|\pi_\mathcal{I}^{-1}(\bfs y_{\mathcal{I}})|\le \delta^2\, q^{m-1}.$$
From this we deduce
\begin{align*}
\sum_{\bfs y_\mathcal{I}\in\pi_\mathcal{I}(W(\F))}
\frac{|\pi_\mathcal{I}^{-1}(\bfs y_{\mathcal{I}})|}{|V(\F)|}
\ln
\frac{|V(\F)|}{|\pi_\mathcal{I}^{-1}(\bfs y_{\mathcal{I}})|}
&\ge \sum_{\bfs y_\mathcal{I}\in\pi_\mathcal{I}(W(\F))}
\frac{|\pi_\mathcal{I}^{-1}(\bfs y_{\mathcal{I}})|}{|V(\F)|}
\ln
\frac{|V(\F)|}{\delta^2\, q^{m-1}}\\
&=
\frac{|W(\F)|}{|V(\F)|}
\ln
\frac{|V(\F)|}{\delta^2\,q^{m-1}}.
\end{align*}

On the other hand, for $\bfs y_{\mathcal{I}}\in
\pi_{\mathcal{I}}((V\setminus W)(\F))$ and any
$j\in\mathcal{E}\setminus \mathcal{I}$, the polynomial
$f_{\mathcal{C}_j}(\bfs y_{\mathcal{I}},X_j)$ is nonzero and has a
finite number of roots in $\overline{\F}$. Thus the fiber
$\pi_{\mathcal{I}}^{-1}(\bfs y_{\mathcal{I}})$ is finite. Moreover,
$$\pi_{\mathcal{I}}^{-1}(\bfs y_{\mathcal{I}})
=V(\F)\cap\{X_i=y_i\ (i\in\mathcal{I})\}\subset
V\cap\{X_i=y_i\ (i\in\mathcal{I})\}.$$
By definition of degree,
$$|\pi_{\mathcal{I}}^{-1}(\bfs y_{\mathcal{I}})|\le
\big|V\cap\{X_i=y_i\ (i\in\mathcal{I})\}\big|\le \delta.$$
Therefore,
\begin{align*}
  \sum_{y_{\mathcal{I}}\in
\pi_{\mathcal{I}}((V\setminus W)(\F))}
\frac{|\pi_{\mathcal{I}}^{-1}(\bfs y_{\mathcal{I}})|}{|V(\F)|}
\ln
\frac{|V(\F)|}{|\pi_{\mathcal{I}}^{-1}(\bfs y_{\mathcal{I}})|} &\ge
  \sum_{y_{\mathcal{I}}\in
\pi_{\mathcal{I}}((V\setminus W)(\F))}
\frac{|\pi_{\mathcal{I}}^{-1}(\bfs y_{\mathcal{I}})|}{|V(\F)|}
\ln
\frac{|V(\F)|}{\delta}\\
&=
\frac{|(V\setminus W)(\F)|}{|V(\F)|}
\ln
\frac{|V(\F)|}{\delta}.
\end{align*}
As a consequence,
\begin{align*}
H(\mathbb{P}_\mathcal{I})&\ge
\frac{|W(\F)|}{|V(\F)|}
\ln
\frac{|V(\F)|}{\delta^2\,q^{m-1}}+
\frac{|(V\setminus W)(\F)|}{|V(\F)|}
\ln
\frac{|V(\F)|}{\delta}\\&=
\ln
(|V(\F)|)- \ln(\delta^2)
-\frac{|W(\F)|}{|V(\F)|}
\ln
(q^{m-1})\\&=
m\ln q +
\ln(1-{5\delta^{7/3}}{q^{-1/2}})- \ln(\delta^2)
-\frac{n\delta^2}{q^{1/2}}
\ln
(q^{m-1}),
\end{align*}
where the lower bound for $\ln (|V(\F)|)$ comes from Lemma
\ref{lemma: explicit Lang-Weil}:
$$
\ln (|V(\F)|)
\ge\ln( q^m-5\delta^{7/3}q^{m-1/2})=
m\ln q +
\ln(1-{5\delta^{7/3}}{q^{-1/2}}).$$
Since $m:=r(\mathcal{E})
=|\mathcal{I}|=r(\mathcal{I})$, we obtain
$$0\ge h_\F(\mathcal{I})-r(\mathcal{I})=\frac{H(\mathbb{P}_\mathcal{I})}{\ln |\F|} - |\mathcal{I}|\ge
\frac{\ln(1-{5\delta^{7/3}}{q^{-1/2}})}{\ln q}- \frac{\ln(\delta^2)}{\ln q}
-\frac{nm\delta^2}{q^{1/2}}.
$$
Summarizing, we have the following result.
\begin{theorem}\label{th: estimate for base}
Let ${\sf M}:=(\mathcal{E},r)$ be an algebraic matroid defined over
a finite field $\G$, let $V\subset\A^n_{\overline{\G}}$ be the
associated variety, and let $\F$ be a finite field extension of $\G$
as in Lemma \ref{lemma: explicit Lang-Weil}. For a base
$\mathcal{I}\subset \mathcal{E}$, with $n:=|\mathcal{E}|$,
$m:=r(\mathcal{E})$, $\delta:=\deg(V)$ and $q:=|\F|$, we have
$$|h_\F(\mathcal{I})-r(\mathcal{I})|\le
-\frac{\ln(1-{5\delta^{7/3}}{q^{-1/2}})}{\ln q}+\frac{\ln(\delta^{2})}{\ln q}
+\frac{nm\delta^{2}}{q^{1/2}}.$$
\end{theorem}

As in \cite{Matus24}, once the case of bases has been established,
the extension to arbitrary independent sets follows by exploiting
the submodularity of entropy.
In our setting, this allows us to transfer the explicit bounds obtained
for bases to any independent set, thereby refining Mat\'u$\check{\mathrm{s}}$’ original
approximation scheme with effective constants depending only on
the degree and dimension of the associated variety $V$.

Let $\mathcal{J}\subset
\mathcal{E}$ be independent but not a base, and let
$\mathcal{I}$ be a base containing $\mathcal{J}$.
By the submodularity of entropy,
\begin{equation}\label{eq: submodularity entropy}
H(\mathbb{P}_\mathcal{I}) \le H(\mathbb{P}_\mathcal{J})
+ H(\mathbb{P}_{\mathcal{I} \setminus \mathcal{J}}).
\end{equation}
As $\mathbb{P}_{\mathcal{I} \setminus \mathcal{J}}$ has support
$\F^{\mathcal{I} \setminus \mathcal{J}}$, it follows that
$H(\mathbb{P}_{\mathcal{I} \setminus \mathcal{J}})\le
\ln|\F^{\mathcal{I} \setminus \mathcal{J}}| = |\mathcal{I}\setminus
\mathcal{J}|\ln|\F|$. Therefore, dividing by $\ln|\F|$ and
subtracting $|\mathcal{I}|$ in \eqref{eq: submodularity entropy},
and using that $r(\mathcal{J}) = |\mathcal{J}|$,
we obtain
\begin{align*}
0\ge \frac{H(\mathbb{P}_\mathcal{J})}{\ln \F}-r(\mathcal{J})&\ge
\frac{H(\mathbb{P}_\mathcal{I})}{\ln \F}-|\mathcal{J}|-|\mathcal{I}\setminus \mathcal{J}|
=h_\F(\mathcal{I})-r(\mathcal{I}).
\end{align*}
From Theorem \ref{th: estimate for base} we deduce the following
corollary.
\begin{corollary}
For any independent set $\mathcal{J} \subset\mathcal{E}$, with hypotheses
and notations as in
Theorem \ref{th: estimate for base}, the following estimate holds:
$$|h_\F(\mathcal{J})-r(\mathcal{J})|\le
-\frac{\ln(1-{5\delta^{7/3}}{q^{-1/2}})}{\ln q}+\frac{\ln(\delta^{2})}{\ln q}
+\frac{nm\delta^{2}}{q^{1/2}}.$$
\end{corollary}


\subsubsection{Dependent sets}
In the previous subsection we obtained upper bounds for the
deviation $|h_\F(\mathcal{I})-r(\mathcal{I})|$ when $\mathcal{I}$ is
an independent set. We now turn to the case of dependent sets.
Following the approach of \cite{Matus24}, we first focus on the
fundamental case of circuits: we establish a sharp bound for the
conditional entropy $H(\mathbb{P}_j|\mathbb{P}_\mathcal{J})$ when
$\{j\}\cup\mathcal{J}$ forms a circuit. This estimate is then used to
control the excess $h_\F(\mathcal{C})-r(\mathcal{C})$ for any circuit
$\mathcal{C}$, and subsequently extended to arbitrary dependent sets.
Our contribution lies in making this strategy quantitative, by
deriving a uniform explicit bound for the deviation
$|h_\F(\mathcal{K})-r(\mathcal{K})|$ valid for every subset
$\mathcal{K}\subset\mathcal{E}$, expressed in terms of $\delta$, $n$,
and $q$.

Consider a circuit $\mathcal{C}\subset\mathcal{E}$.
According to Proposition \ref{prop: degree bound f_C}, there exists a
nonzero absolutely irreducible polynomial $f_\mathcal{C} \in \G[X_i
: i \in \mathcal{C}]$ of degree at most $\delta$ such that
\[
f_\mathcal{C}((f_i)_{i \in \mathcal{C}}) = 0.
\]
For $j\in\mathcal{C}$, $f_\mathcal{C}$ may be viewed as a polynomial
in the indeterminate $X_j$ of degree $d_{j,\mathcal{C}}\le \delta$,
with coefficients in $\G[X_i:i\in \mathcal{J}]$,
where $\mathcal{J}:= \mathcal{C}\setminus\{j\}$.
In this setting we have
\[
H(\mathbb{P}_\mathcal{C}) =
H(\mathbb{P}_\mathcal{J},\mathbb{P}_{\{j\}})=
H(\mathbb{P}_\mathcal{J})+H(\mathbb{P}_{\{j\}}|\mathbb{P}_\mathcal{J}).
\]
We denote $\mathbb{P}_j:=\mathbb{P}_{\{j\}}$. By definition,
\begin{align*}
&H(\mathbb{P}_j|\mathbb{P}_\mathcal{J})\\\ &=-\sum_{\bfs y_\mathcal{J}\in\pi_\mathcal{J}(V(\F))}
\mathbb{P}_\mathcal{J}(S_\mathcal{J}=\bfs y_\mathcal{J})
\sum_{y_j\in\pi_j(V(\F))}
\mathbb{P}_j(S_j=y_j|S_\mathcal{J}=\bfs y_\mathcal{J})\ln \mathbb{P}_j(S_j=y_j|S_\mathcal{J}=\bfs y_\mathcal{J})\\
\ &=-\sum_{\bfs y_\mathcal{C}\in\pi_\mathcal{C}(V(\F))}
\mathbb{P}_\mathcal{C}(S_\mathcal{C}=\bfs y_\mathcal{C})
\ln \frac{\mathbb{P}_\mathcal{C}(S_\mathcal{C}=\bfs y_\mathcal{C})}
{\mathbb{P}_\mathcal{J}(S_\mathcal{J}=\pi_\mathcal{J}(\bfs y_\mathcal{C}))}
\end{align*}

Observe that
$$
\mathbb{P}_\mathcal{C}(S_\mathcal{C}=\bfs y_\mathcal{C})=
\frac{|\pi_\mathcal{C}^{-1}(\bfs y_\mathcal{C})|}{|V(\F)|},\quad
{\mathbb{P}_\mathcal{J}(S_\mathcal{J}=\pi_\mathcal{J}(\bfs y_\mathcal{C}))}=
\frac{|\pi_\mathcal{J}^{-1}(\pi_\mathcal{J}(\bfs y_\mathcal{C}))|}{|V(\F)|}.
$$
Therefore,
\begin{align}
H(\mathbb{P}_j|\mathbb{P}_\mathcal{J})
&=-\sum_{\bfs y_\mathcal{C}\in\pi_\mathcal{C}(V(\F))}
\frac{|\pi_\mathcal{C}^{-1}(\bfs y_\mathcal{C})|}{|V(\F)|}\ln
\frac{|\pi_\mathcal{C}^{-1}(\bfs y_\mathcal{C})|}{|\pi_\mathcal{J}^{-1}(\pi_\mathcal{J}(\bfs y_\mathcal{C}))|}
\notag\\
&=\sum_{\bfs y_\mathcal{J}\in\pi_\mathcal{J}(V(\F))}
\frac{|\pi_\mathcal{J}^{-1}(\bfs y_\mathcal{J})|}{|V(\F)|}
\sum_{y_j\in\pi_j(V(\F))} \frac{|\pi_\mathcal{C}^{-1}(\bfs
y_\mathcal{J},y_j)|}{|\pi_\mathcal{J}^{-1}(\bfs y_\mathcal{J})|} \ln
\frac{|\pi_\mathcal{J}^{-1}(\bfs
y_\mathcal{J})|}{|\pi_\mathcal{C}^{-1}(\bfs y_\mathcal{J},y_j)|}.
\label{eq: first expression conditional entropy}
\end{align}

To bound $H(\mathbb{P}_j|\mathbb{P}_\mathcal{J})$, we analyze the
inner sum for each $\bfs y_\mathcal{J}$, distinguishing whether the
specialization $f_\mathcal{C}(\bfs y_\mathcal{J},X_j)$ vanishes
identically or not. Denote by $W_{j,\mathcal{C}}$ the set of $\bfs
x\in V$ for which $f_\mathcal{C}(\pi_\mathcal{J}(\bfs x),X_j)$ is
the zero polynomial.
\begin{lemma}\label{lemma: upper bound inner sum}
For every $\bfs y_\mathcal{J}\in\pi_\mathcal{J}(V(\F))$, with
$d_{j,\mathcal{C}}:=\deg f_{\mathcal{C}_j}$, we have
$$
\sum_{y_j\in\pi_j(V(\F))} \frac{|\pi_\mathcal{C}^{-1}(\bfs
y_\mathcal{J},y_j)|}{|\pi_\mathcal{J}^{-1}(\bfs y_\mathcal{J})|} \ln
\frac{|\pi_\mathcal{J}^{-1}(\bfs
y_\mathcal{J})|}{|\pi_\mathcal{C}^{-1}(\bfs y_\mathcal{J},y_j)|}\le
\left\{\begin{array}{l}
\ln q \textrm{\quad for }\bfs y_\mathcal{J}\in\pi_\mathcal{J}(W_{j,\mathcal{C}}(\F)),  \\[0.3em]
 \ln d_{j,\mathcal{C}}  \textrm{\quad otherwise}.
\end{array}\right.
$$
\end{lemma}
\begin{proof}
If $\bfs y_\mathcal{J}\in \pi_\mathcal{J}(W_{j,\mathcal{C}}(\F))$,
then the polynomial $f_\mathcal{C}(\bfs y_\mathcal{J},X_j)$ vanishes
identically and thus the fiber $\pi_\mathcal{J}^{-1}(\bfs
y_\mathcal{J})$ decomposes as the disjoint union
$$\pi_\mathcal{J}^{-1}(\bfs y_\mathcal{J})=\bigcup_{y_j\in\F}\pi_\mathcal{C}^{-1}(\bfs y_\mathcal{J},y_j),$$
where all the sets in the right-hand side have the same cardinality.
As a consequence,
$$|\pi_\mathcal{J}^{-1}(\bfs y_\mathcal{J})|=\sum_{y_j\in\F}|\pi_\mathcal{C}^{-1}(\bfs y_\mathcal{J},y_j)|
=q\,|\pi_\mathcal{C}^{-1}(\bfs y_\mathcal{J},y_j)|,$$
for any $y_j\in\F$, which readily implies
$$\sum_{y_j\in\pi_j(V(\F))}
\frac{|\pi_\mathcal{C}^{-1}(\bfs
y_\mathcal{J},y_j)|}{|\pi_\mathcal{J}^{-1}(\bfs y_\mathcal{J})|} \ln
\frac{|\pi_\mathcal{J}^{-1}(\bfs
y_\mathcal{J})|}{|\pi_\mathcal{C}^{-1}(\bfs y_\mathcal{J},y_j)|}=
\sum_{y_j\in\pi_j(V(\F))} \frac{|\pi_\mathcal{C}^{-1}(\bfs
y_\mathcal{J},y_j)|}{|\pi_\mathcal{J}^{-1}(\bfs y_\mathcal{J})|} \ln
q=\ln q.
$$

On the other hand, if $\bfs y_\mathcal{J}\notin
\pi_\mathcal{J}(W_{j,\mathcal{C}}(\F))$, then $f_\mathcal{C}(\bfs
y_\mathcal{J},X_j)$ has at most $d_{j,\mathcal{C}}$ roots in $\F$.
Hence the fiber $\pi_\mathcal{J}^{-1}(\bfs y_\mathcal{J})$
decomposes as the disjoint union
$$\pi_\mathcal{J}^{-1}(\bfs y_\mathcal{J})=\bigcup_{y_j\in\F:f_\mathcal{C}(\bfs y_\mathcal{J},y_j)=0}
\pi_\mathcal{C}^{-1}(\bfs y_\mathcal{J},y_j).$$
If $A_{\bfs y_\mathcal{J}}:=\{y_j\in\F:f_\mathcal{C}(\bfs y_\mathcal{J},y_j)=0\}$, then
$|A_{\bfs y_\mathcal{J}}|\le d_{j,\mathcal{C}}$, and it follows that
\begin{align*}
  \sum_{y_j\in\pi_j(V(\F))}
&\frac{|\pi_\mathcal{C}^{-1}(\bfs
y_\mathcal{J},y_j)|}{|\pi_\mathcal{J}^{-1}(\bfs y_\mathcal{J})|} \ln
\frac{|\pi_\mathcal{J}^{-1}(\bfs
y_\mathcal{J})|}{|\pi_\mathcal{C}^{-1}(\bfs y_\mathcal{J},y_j)|} =
  \sum_{y_j\in A_{\bfs y_\mathcal{J}}}
\frac{|\pi_\mathcal{C}^{-1}(\bfs y_\mathcal{J},y_j)|}{|\pi_\mathcal{J}^{-1}(\bfs y_\mathcal{J})|}
\ln
\frac{|\pi_\mathcal{J}^{-1}(\bfs y_\mathcal{J})|}{|\pi_\mathcal{C}^{-1}(\bfs y_\mathcal{J},y_j)|}.
\end{align*}
This sum is maximized when all the fibers $\pi_\mathcal{C}^{-1}(\bfs
y_\mathcal{J},y_j)$ have the same cardinality, in which case is upper
bounded by the logarithm of the number of
summands. As this number is at most $d_{j,\mathcal{C}}$, we easily
conclude that
$$  \sum_{y_j\in\pi_j(V(\F))}
\frac{|\pi_\mathcal{C}^{-1}(\bfs
y_\mathcal{J},y_j)|}{|\pi_\mathcal{J}^{-1}(\bfs y_\mathcal{J})|} \ln
\frac{|\pi_\mathcal{J}^{-1}(\bfs
y_\mathcal{J})|}{|\pi_\mathcal{C}^{-1} (\bfs y_\mathcal{J},y_j)|}
\le\ln d_{j,\mathcal{C}},$$
as claimed.
\end{proof}

Now we proceed to obtain an upper bound for $H(\mathbb{P}_j|\mathbb{P}_\mathcal{J})$.
\begin{proposition}\label{prop: bound condition entropy}
Suppose that $q:=|\F|>(5\delta^{7/3}+1)^2$. Let $\mathcal{C}$ be a
circuit of ${\sf M}:=(\mathcal{E},r)$. Let $j\in\mathcal{C}$ and
$\mathcal{J}:=\mathcal{C}\setminus\{j\}$. Then
$$H(\mathbb{P}_j|\mathbb{P}_\mathcal{J})\le  \frac{\delta^{2}}{q^{1/2}}\ln q+\ln\delta. $$
\end{proposition}
\begin{proof}
From \eqref{eq: first expression conditional entropy} we have
\begin{align*}
H(\mathbb{P}_j|\mathbb{P}_\mathcal{J})   &= \sum_{\bfs
y_\mathcal{J}\in\pi_\mathcal{J}(V(\F))}
\frac{|\pi_\mathcal{J}^{-1}(\bfs y_\mathcal{J})|}{|V(\F)|}
\sum_{y_j\in\pi_j(V(\F))} \frac{|\pi_\mathcal{C}^{-1}(\bfs
y_\mathcal{J},y_j)|}{|\pi_\mathcal{J}^{-1}(\bfs y_\mathcal{J})|} \ln
\frac{|\pi_\mathcal{J}^{-1}(\bfs
y_\mathcal{J})|}{|\pi_\mathcal{C}^{-1}(\bfs y_\mathcal{J},y_j)|}.
\end{align*}
Since
$\pi_\mathcal{J}(V(\F))=\pi_\mathcal{J}\big(W_{j,\mathcal{C}}(\F)\big)\,\cup\,
\pi_\mathcal{J}\big((V\setminus W_{j,\mathcal{C}})(\F)\big)$, we can
decompose the sum into two parts. Applying Lemma \ref{lemma: upper
bound inner sum}, we obtain
\begin{align*}
H(\mathbb{P}_j|\mathbb{P}_\mathcal{J})&\le
\sum_{\bfs y_\mathcal{J}\in \pi_\mathcal{J}(W_{j,\mathcal{C}}(\F))}
\frac{|\pi_\mathcal{J}^{-1}(\bfs y_\mathcal{J})|}{|V(\F)|}\ln q +
\sum_{\bfs y_\mathcal{J}\in\pi_\mathcal{J}((V\setminus W_{j,\mathcal{C}})(\F))}
\frac{|\pi_\mathcal{J}^{-1}(\bfs y_\mathcal{J})|}{|V(\F)|}\ln d_{j,\mathcal{C}}\\
&\le \frac{|W_{j,\mathcal{C}}(\F)|}{|V(\F)|}\ln q + \ln
d_{j,\mathcal{C}}
\end{align*}

Next we bound $|W_{j,\mathcal{C}}(\F)|$. Recall that
$W_{j,\mathcal{C}}$ is the set of $\bfs x\in V$ for which
$f_\mathcal{C}(\pi_\mathcal{J}(\bfs x),X_j)=0$. The polynomial
$f_\mathcal{C}\in \G[X_i:i\in \mathcal{C}]$ is absolutely
irreducible and lies in $\G[(X_i:i\in
\mathcal{J}][X_j]\setminus\G[X_i:i\in \mathcal{J}]$.
Writing
$$f_{\mathcal{C}}=\sum_{h\in\mathcal{H}}m_hX_j^h,$$
with $m_h\in\G[(X_i)_{i\in \mathcal{J}}]$ nonzero for all $h$, we
note that the algebraic independence of $\{f_i:i\in \mathcal{J}\}$
implies $m_h((f_i)_{i\in \mathcal{J}})\not=0$. Hence $m_h\notin
I_\G$, and since $V$ is absolutely irreducible, $V\cap \{m_h=0\}$
has pure codimension one in $V$ or is empty. Therefore
$$W_{j,\mathcal{C}}\subset V\cap \{m_h=0\}.$$
Applying \eqref{eq: upper bound -- affine gral}, together with
\eqref{eq: Bezout} and Proposition \ref{prop: degree bound f_C}, we
obtain
\begin{align*}
|W_{j,\mathcal{C}}(\F)|\le
|(V\cap \{m_h=0\})(\F)|
&\le \deg (V\cap \{m_h=0\})
q^{\dim V-1}\\
&\le \deg (V)\deg(m_h)\,q^{\dim V-1}\le \delta^{2}\,q^{m-1}.
\end{align*}

Finally, by Lemma \ref{lemma: explicit Lang-Weil},
$$
|V(\F)|\ge
q^m-5\delta^{7/3}q^{m-1/2}.
$$
The assumption on $q$ implies that the right-hand side is greater
than $q^{m-1/2}$ (see Lemma \ref{lemma: explicit Lang-Weil}). We
thus conclude
\begin{align*}
H(\mathbb{P}_j|\mathbb{P}_\mathcal{J})&\le \frac{|
W_{j,\mathcal{C}}(\F)|}{|V(\F)|}\ln q + \ln d_{j,\mathcal{C}}\\&\le
\frac{\delta^{2}}{q-5\delta^{7/3}q^{1/2}}\ln q+\ln\delta \le
\frac{\delta^{2}}{q^{1/2}}\ln q+\ln \delta,
\end{align*}
as claimed.
\end{proof}

Now we are able to obtain an upper bound on the difference
$h_\F(\mathcal{C})-r(\mathcal{C})$.
\begin{corollary}\label{coro: bound diff entropy}
With the assumptions of Proposition \ref{prop: bound condition
entropy}, and with $h_\F:2^\mathcal{E}\to\mathbb{R}$ defined as in
\eqref{eq: definition polymatroid}, for any circuit $\mathcal{C}$ we
have
$$h_\F(\mathcal{C})-r(\mathcal{C})\le \frac{\delta^2}{q^{1/2}}+\frac{\ln\delta}{\ln q}.$$
\end{corollary}
\begin{proof}
Let $\mathcal{J}:=\mathcal{C}\setminus\{j\}$ for a given
$j\in\mathcal{C}$. By definition,
$h_\F(\mathcal{A}):=\frac{H(\mathbb{P}_\mathcal{A})}{\ln|\F|}$ for
any $\mathcal{A}\subset\mathcal{E}$. For a circuit $\mathcal{C}$ and
$\mathcal{J}$ as above, we have
$$
H(\mathbb{P}_\mathcal{C})-H(\mathbb{P}_\mathcal{J})=H(\mathbb{P}_j|\mathbb{P}_\mathcal{J})
\le \frac{\delta^2}{q^{1/2}}\ln q+\ln\delta,
$$
by Proposition \ref{prop: bound condition entropy}. Dividing by $\ln
q$ yields
\begin{equation}\label{eq: bound diff entropies}
h_\F(\mathcal{C})-h_\F(\mathcal{J})
\le \frac{\delta^2}{q^{1/2}}+\frac{\ln\delta}{\ln q}.
\end{equation}
Finally,
$$h_\F(\mathcal{C})-\bigg(\frac{\delta^2}{q^{1/2}}
+\frac{\ln\delta}{\ln q}\bigg)\le
h_\F(\mathcal{J})=\frac{H(\mathbb{P}_\mathcal{J})}{\ln q}\le
\frac{\ln|\F^\mathcal{J}|}{\ln q}=|\mathcal{J}|=r(\mathcal{C}),$$
which proves the claim.
\end{proof}

Next we obtain a corresponding bound for an arbitrary dependent set.
\begin{lemma}\label{lemma: bound entropy arbitrary K}
With the assumptions of Proposition \ref{prop: bound condition
entropy}, for any dependent set $\mathcal{K}$ of $\mathcal{E}$ we
have
$$h_\F(\mathcal{K})-r(\mathcal{K})\le |\mathcal{E}|
\bigg(\frac{\delta^2}{q^{1/2}}+\frac{\ln\delta}{\ln q}\bigg).$$
\end{lemma}
\begin{proof}
Let $\mathcal{K}$ be a dependent subset of $\mathcal{E}$, and let
$\mathcal{J}$ be a maximal independent set contained in
$\mathcal{K}$. For $k\in \mathcal{K}\setminus \mathcal{J}$, let
$\gamma(k, \mathcal{J})$ be the unique circuit contained in $k\cup
\mathcal{J}$. By the submodularity of the entropy,
$$h_\F(\gamma(k, \mathcal{J}))-h_\F(\gamma(k, \mathcal{J})\setminus \{k\})
\ge h_\F(\mathcal{J}\cup \{k\})-h_\F(\mathcal{J}).$$
Setting $\{k_1, \dots, k_t\}:=\mathcal{K} \setminus \mathcal{J} $
with $t \ge 1$, and considering the chain
\[
\mathcal{J} \subset \mathcal{J} \cup \{k_1\} \subset \mathcal{J}
\cup \{k_1,k_2\} \subset \dots \subset \mathcal{J} \cup \{k_1,\dots,k_t\} = \mathcal{K},
\]
by the submodularity of entropy we deduce
\begin{align*}
h_\F(\mathcal{K}) - h_\F(\mathcal{J})
&=\sum_{i=1}^t\big(h_\F(\mathcal{K}\setminus\{k_1,\ldots,k_{i-1}\}) - h_\F(\mathcal{K}\setminus\{k_1,\ldots,k_i\})
\big)\\
&\le \sum_{i=1}^t\big(h_\F(\mathcal{J}\cup \{k_i\}) - h_\F(\mathcal{J})\big)\\
&\le \sum_{i=1}^t\big(h_\F(\gamma(k_i,\mathcal{J})) -
h_\F(\gamma(k_i,\mathcal{J})\setminus\{k_i\})\big).
\end{align*}
Applying \eqref{eq: bound diff entropies} to each term yields
$$h_\F(\mathcal{K}) - h_\F(\mathcal{J})
\le t\bigg(\frac{\delta^2}{q^{1/2}}+\frac{\ln\delta}{\ln q}\bigg)
\le |\mathcal{E}|\bigg(\frac{\delta^2}{q^{1/2}}+\frac{\ln\delta}{\ln
q}\bigg).$$
Thus
\begin{equation}\label{eq: bound entropy arbitrary K}
h_\F(\mathcal{K}) - |\mathcal{E}|\bigg(\frac{\delta^2}{q^{1/2}}+\frac{\ln\delta}{\ln q}\bigg)
\le h_\F(\mathcal{J}) \le |\mathcal{J}| =r(\mathcal{K}).
\end{equation}
This implies the lemma.
\end{proof}

Now we derive a corresponding lower bound for
$h_\F(\mathcal{K})-r(\mathcal{K})$. Observe that
$h_\F(\mathcal{E})=\frac{\ln |V(\F)|}{\ln q}$. By Lemma \ref{lemma:
explicit Lang-Weil},
$$
|V(\F)| \ge
q^m-5\delta^{7/3}q^{m-1/2}=q^m\bigg(1-\frac{5\delta^{7/3}}{q^{1/2}}\bigg),$$
so that
$$h_\F(\mathcal{E})= \frac{\ln |V(\F)|}{\ln q}\ge m +
\frac{\ln(1-{5\delta^{7/3}}{q^{-1/2}})}{\ln q}.$$
Applying \eqref{eq: bound entropy arbitrary K} with $\mathcal{K} =
\mathcal{E}$ and $\mathcal{I}$ a maximal independent subset, for any
$\mathcal{J}\subset \mathcal{I}$, we have $h_\F(\mathcal{I}) \le
h_\F(\mathcal{J}) + h_\F(\mathcal{I} \setminus \mathcal{J})$, and
therefore
\begin{align*}
m+
\frac{\ln(1-{5\delta^{7/3}}{q^{-1/2}})}{\ln q}
\le h_\F(\mathcal{E}) &\le h_\F(\mathcal{I}) + |\mathcal{E}|\bigg(\frac{\delta^2}{q^{1/2}}
+\frac{\ln\delta}{\ln q}\bigg)
\\&\le h_\F(\mathcal{J}) + |\mathcal{I} \setminus \mathcal{J}| +
|\mathcal{E}|\bigg(\frac{\delta^2}{q^{1/2}}
+\frac{\ln\delta}{\ln q}\bigg),
\end{align*}
since $\frac{H(\mathbb{P}_{\mathcal{I}\setminus \mathcal{J}})}{\ln
q} \le \frac{\ln |\F^{\mathcal{I}\setminus \mathcal{J}}|}{\ln q}=
|\mathcal{I}\setminus \mathcal{J}|$. Hence
\begin{equation}\label{eq: auxiliar lower bound approx}
m - |\mathcal{I} \setminus \mathcal{J}| +
\frac{\ln(1-{5\delta^{7/3}}{q^{-1/2}})}{\ln q}- |\mathcal{E}|
\bigg(\frac{\delta^2}{q^{1/2}}
+\frac{\ln\delta}{\ln q}\bigg)
\le  h_\F(\mathcal{J}).
\end{equation}

Finally, let $\mathcal{K}$ be any dependent set, let $\mathcal{J}$
be a maximal independent subset of $\mathcal{K}$, and let
$\mathcal{I}$ be a maximal independent subset of $\mathcal{E}$
containing $\mathcal{J}$. Since $r(\mathcal{E}) - |\mathcal{I}
\setminus \mathcal{J}|=r(\mathcal{K})$, from \eqref{eq: auxiliar
lower bound approx} we deduce
$$r(\mathcal{K}) +
\frac{\ln(1-{5\delta^{7/3}}{q^{-1/2}})}{\ln q}- |\mathcal{E}|\bigg(\frac{\delta^2}{q^{1/2}}
+\frac{\ln\delta}{\ln q}\bigg)\le h_\F(\mathcal{J}) \le h_\F(\mathcal{K}),$$
which gives the desired lower bound.

Combining Lemma \ref{lemma: bound entropy arbitrary K} with the
lower bound derived from Lemma \ref{lemma: explicit Lang-Weil},
we obtain the following quantitative refinement of Mat\'u$\check{\textrm{s}}$'s
approximation scheme.
\begin{theorem}
Let ${\sf M}:=(\mathcal{E},r)$ be an algebraic matroid defined over
a finite field $\G$, let $V\subset\A^n_{\overline{\G}}$ be the variety
associated to ${\sf M}$, and let $\F$ be a finite field extension of $\G$
as in Lemma \ref{lemma: explicit Lang-Weil}. For an arbitrary subset
$\mathcal{K}$ of $\mathcal{E}$, with $n:=|\mathcal{E}|$,
$\delta:=\deg(V)$ and $q:=|\F|$, we have
$$|h_\F(\mathcal{K}) - r(\mathcal{K})|\le
\frac{n\delta^2}{q^{1/2}}+\frac{n\ln\delta}{\ln q}
-\frac{\ln(1-{5\delta^{7/3}}{q^{-1/2}})}{\ln q}.
$$
\end{theorem}

This theorem shows that the qualitative scheme of
\cite{Matus24}---first controlling circuits and then extending to
arbitrary dependent sets--—admits an explicit quantitative version.
Our contribution lies in providing effective constants depending
only on $n$, $\delta$, and $q$, thereby turning the almost entropic
approximation of algebraic matroids into a concrete, computable
bound.

%
%

\section{Conclusions}

We established explicit, uniform bounds comparing the entropy
polymatroid $h_\F$ of the evaluation distribution on $V(\F)$ with
the rank function $r$ of the algebraic matroid ${\sf
M}:=(\mathcal{E},r)$ under consideration, where
$V\subset\A^n_{\overline{\G}}$ is the associated variety and $\F$ is
a finite extension of the finite field $\G$ of definition of ${\sf
M}$. In particular, for every $\mathcal{K}\subset\mathcal{E}$,
\[
\bigl|h_\F(\mathcal{K})-r(\mathcal{K})\bigr| \;\le\;
\frac{nm\delta^{2}}{q^{1/2}} \;+\; \frac{nm\ln\delta^2}{\ln q} \;-\;
\frac{\ln\!\bigl(1-5\delta^{7/3}q^{-1/2}\bigr)}{\ln q}.
\]
Hence, as $q:=|\F|\to\infty$, the discrepancy tends to $0$ with rate
$\mathcal{O}(1/\ln q)$, with constants depending only on
$n:=|\mathcal{E}|$ and $\delta:=\deg(V)$. The argument proceeds by
first controlling independent sets, then circuits via a conditional
entropy bound, and finally extending to arbitrary dependent sets by
submodularity.


\providecommand{\bysame}{\leavevmode\hbox
to3em{\hrulefill}\thinspace}
\providecommand{\MR}{\relax\ifhmode\unskip\space\fi MR }
\providecommand{\MRhref}[2]{%
  \href{http://www.ams.org/mathscinet-getitem?mr=#1}{#2}
} \providecommand{\href}[2]{#2}

\end{document}

\end{document}